\documentclass[12pt,a4paper]{article}
\usepackage[utf8]{inputenc}
\usepackage{amsmath}
\usepackage{amsfonts}
\usepackage{amssymb}
\usepackage{amsthm}

\usepackage{comment}
\usepackage{url}

\usepackage{tikz}
\usetikzlibrary{arrows}
\usetikzlibrary{matrix}

\newcommand{\Q}{\mathbb{Q}}
\newcommand{\R}{\mathbb{R}}

\newcommand{\Hom}{\operatorname{Hom}}

\newcommand{\e}{\varepsilon}
\newcommand{\de}{\delta}

\theoremstyle{plain}
\newtheorem{theorem}{Theorem}
\newtheorem*{theorem*}{Theorem}
\newtheorem{lemma}{Lemma}

\newtheorem{corollary}{Corollary}

\theoremstyle{definition}
\newtheorem*{example}{Example}
\newtheorem{definition}{Definition}
\newtheorem*{remark}{Remark}

\newcommand{\colim}{\operatorname{colim}}

\newcommand{\ds}{\displaystyle}

\newcommand{\Cat}{\mathcal{C}}

\title{\bf{Metric Limits in Categories with a Flow}}
\author{Joshua Cruz}
\date{}
\begin{document}
\maketitle
\begin{abstract} 
In topological data science, categories with a flow have become ubiquitous, including as special cases examples like persistence modules and sheaves. With the flow comes an interleaving distance, which has proven useful for applications. We give simple, categorical conditions which guarantee metric completeness of a category with a flow, meaning that every Cauchy sequence has a limit. We also describe how to find a metric completion of a category with a flow by using its Yoneda embedding. The overarching goal of this work is to prepare the way for a theory of convergence of probability measures on these categories with a flow.
\end{abstract}

\section{Introduction}\label{sec:Introduction}

It is common in applied topology to take a data set and assign to it an object from some category. Examples include persistence modules \cite{Chazal09}, multiparameter persistence modules \cite{Lesnick}, and derived sheaves \cite{KS}; see \cite{MunchCatsFlow} for more examples of categories used in this way. If this category can be given an interleaving distance, we can often treat it as a metric space (up to some technical complications).\\

A general way to get an interleaving distance is to have a flow as defined in \cite{MunchCatsFlow}. On a category $\Cat$, a flow is a functor $T_\e : \Cat\to \Cat$ for every $\e \in [0,\infty)$ satisfying certain properties. It is natural to ask whether $(\Cat, T_\e)$ is metrically complete\footnote{It is inconvenient that many terms we would like to use (such as \emph{limit}, \emph{complete}, \emph{dense}, and even \emph{Cauchy complete category}) have already be defined by category theorists in other contexts; see \cite{Lawvere,CauchyCompleteCatTheory}. We will try to avoid confusion by using the adjective \emph{metric} or adverb \emph{metrically} to denote terms related to limits coming from the interleaving distance and by using the adjective \emph{categorical} to refer to the category theory versions.}; i.e. if every Cauchy sequence has a metric limit. In particular, one might hope that there are categorical conditions we can place on $\Cat$ to ensure it is. \\

In this paper, we study the issue of metric completeness of categories with a flow. We are particularly interested in the connection between categorical properties of $\Cat$ and $T_\e$ and their metric properties. The following is an example of the kind of theorem that is proved:

\begin{theorem*} Let $(\Cat, T_\e)$ be a category with a flow. Then $(\Cat, T_\e)$ is metrically complete if $\Cat$ is categorically complete and $T_\e$ preserves limits.
\end{theorem*}

These conditions are in fact stronger than what we actually need; see Theorem \ref{thm:CategorialConditionsForCompleteness}. The proof involves constructing a diagram out of a given Cauchy sequence; then (under conditions on $T_\e$), the categorical limit of the diagram is a metric limit of the Cauchy sequence. We also describe the closure of a subcategory with a strict flow in a larger, complete category. This gives an analog of Cauchy completion using the (co-)Yoneda Embedding. \\

We can use this general framework to study completeness for many specific examples of categories with interleavings.  Generalized persistence modules and derived sheaves are two more new-to-the-literature contexts in which we address completeness. \\

Let's give some indication of the motivation for this work. One might try to study convergence of probability measures directly on these categories with interleavings. Lots of previous work has studied stochastic processes at the level of simplicial complexes (e.g. Erd\"os-Renyi simplicial complexes) and at the level of persistence diagrams, persistence landscapes, betti numbers, and other invariants. However, working at the intermediate level of categories with interleaving distances would be preferable in many contexts, especially when a sufficiently descriptive and well-behaved invariant has not been found. Multi-parameter persistence modules come to mind. Very little work on the convergence of probability measures has been done at this level, in part because there does not yet exist strong foundations for this study. What is needed is an understanding at least of metric completion and separability of the category, and hopefully also some idea of the precompact subsets\footnote{Perhaps we should say precompact subcategories.}. Prokhorov's Theorem is a prime example of these three concepts in play.\\

While the majority of this paper is focused around studying the metric completeness of a category with a flow, the last section situates this paper in the broader context of Polish spaces and their applications to the convergence of probability measures.

\subsection*{Overview}

Section \ref{sec:Background} gives background and terminology for categories with a flow, heavily influenced by \cite{MunchCatsFlow}. It also introduces the dual notion of categories with a coflow. Section \ref{sec:Main} defines Cauchy sequences in categories with flows and gives categorical conditions on when they have metric limits. Section \ref{sec:Completions} explains how to use the Yoneda embedding to densely embed any category with a strict flow into a metrically complete category. Section \ref{sec:Examples} studies some example categories and shows whether they are complete or not. Section \ref{sec:Polish} gives some indication of the context and broader interest of these results by giving an application towards finding categories with a flow which are Polish spaces.

\subsection*{Acknowledgements}

I'd like to acknowledge and thank the IMA and the CMO for hosting conferences this year which gave me access to several very useful conversations. In particular, the idea for this paper was conceived at the IMA's conference ``Bridging Sheaves and Statistics''. At both these conferences, many people were very helpful. In particular, Justin Curry, Peter Bubenik, and Nikola Milicevic all gave important observations and pointers that made this paper much better.

\section{Background on Categories with Flows}\label{sec:Background}

The interleaving distance was first defined in the context of persistence modules \cite{Chazal09}. It has since been generalized in many different ways. Lesnick \cite{Lesnick} defines an interleaving distance on multiparameter persistence modules. Bubenik, de Silva, and Scott \cite{Bub15} define an interleaving distance for generalized persistence modules, which are functor categories $[\mathcal{P},\mathcal{D}]$, where $\mathcal{D}$ is any category and $\mathcal{P}$ is a preordered set. Kashiwara and Schapira \cite{KS} define an interleaving distance for constructible derived sheaves on $\R^n$.\\

De Silva, Munch, and Stefanou \cite{MunchCatsFlow,StefanouThesis} make a definition which generalizes all other current examples, at least to the author's knowledge. They define a \emph{category with a flow}, and show that this induces an interleaving distance with the required properties, and that this distance agrees with the proposed distance in many other circumstances. Further, they were able to provide a very general stability theorem using a notion of flow-equivariant functors.

\subsection{Definitions and Main Results}

All the definitions and results from this subsection can be found in \cite{MunchCatsFlow}.

\begin{definition} A flow on a category $\Cat$ is an assignment of a endofunctor $T_\e : \Cat \to \Cat$ for every non-negative real number $\e\geq 0$ along with natural transformations $T_{\de \leq \e}: T_\de \Rightarrow T_\e$ for $\de \leq \e$ so that for all $\de\leq \e\leq \zeta$, the following commutes
\begin{center}
\begin{tikzpicture}
\matrix (m) [matrix of math nodes,row sep=3em,column sep=4em,minimum width=2em]
{
T_\de  &  & T_{\zeta}\\
& T_{\e}&\\
};
\path[-stealth]
    (m-1-1) edge node [above] {$T_{\de \leq \zeta}$} (m-1-3)
    		edge node[xshift=2pt,yshift=-5pt] [left] {$T_{\de \leq \e}$ } (m-2-2)		
    (m-2-2) edge node[xshift=1pt,yshift=-5pt] [right] {$T_{\e \leq \zeta}$ } (m-1-3);
\end{tikzpicture}
\end{center}

Said another way, a flow is a functor $T: [0,\infty) \to \mathbf{End}(\Cat)$ from the poset category of non-negative reals to the category of endofunctors of $\Cat$. Further, we require any flow to come with natural transformations $\mu_{\e, \de}: T_\e T_\de \Rightarrow T_{\e + \de}$ and $u: Id_{\Cat} \Rightarrow T_0$ satifying the following relations:
\begin{center}
\begin{tikzpicture}
\matrix (m) [matrix of math nodes,row sep=3em,column sep=4em,minimum width=2em]
{
 & T_\e  &  \\
T_0 T_\e & & T_\e\\
};
\path[-stealth]
   
    (m-1-2) edge[double, double distance=2pt, -] node [right] { } (m-2-3)		
    		edge node[xshift=5pt,yshift=10pt] [left] {$u I_{T_\e}$ } (m-2-1)
	(m-2-1) edge node [below] {$\mu_{0, \e}$} (m-2-3);
            
\end{tikzpicture}
\hspace{0.5in}
\begin{tikzpicture}
\matrix (m) [matrix of math nodes,row sep=3em,column sep=4em,minimum width=2em]
{
 & T_\e  &  \\
T_\e T_0 & & T_\e\\
};
\path[-stealth]
   
    (m-1-2) edge[double, double distance=2pt, -] node [right] { } (m-2-3)		
    		edge node[xshift=5pt,yshift=10pt] [left] {$I_{T_\e} u$ } (m-2-1)
	(m-2-1) edge node [below] {$\mu_{0, \e}$} (m-2-3);
            
\end{tikzpicture}
\end{center}

\begin{center}
\begin{tikzpicture}
\matrix (m) [matrix of math nodes,row sep=3em,column sep=4em,minimum width=2em]
{
T_\e T_\zeta T_\de   &  & T_\e T_{\zeta + \de} \\
T_{\e + \zeta} T_\de & & T_{\e + \zeta + \de}\\
};
\path[-stealth]
    (m-1-1) edge node [above] {$ Id_{T_\e}\mu_{\zeta,\de }$ } (m-1-3)
    		edge node [left] {$\mu_{\e, \zeta} Id_{T_\de}$ } (m-2-1)	
    (m-1-3) edge node [right] {$\mu_{\e, \zeta + \de}$} (m-2-3)	
    (m-2-1) edge node [below] {$\mu_{\e + \zeta, \de}$ } (m-2-3);
\end{tikzpicture}
\hspace{0.5in}
\begin{tikzpicture}
\matrix (m) [matrix of math nodes,row sep=3em,column sep=4em,minimum width=2em]
{
T_\e T_\zeta  &  & T_{\e + \zeta} \\
T_\de T_\kappa & & T_{\de + \kappa}\\
};
\path[-stealth]
    (m-1-1) edge node [above] {$ \mu_{\e, \zeta}$ } (m-1-3)
    		edge node [left] {$T_{\e \leq \de}T_{\zeta \leq \kappa}$ } (m-2-1)	
    (m-1-3) edge node [right] {$T_{\e + \zeta \leq \de + \kappa}$} (m-2-3)	
    (m-2-1) edge node [below] {$\mu_{\de,\kappa}$ } (m-2-3);
\end{tikzpicture}
\end{center}

A flow is \emph{strict} if the natural transformations $Id_\Cat \Rightarrow T_0$ and $T_\e T_\de \Rightarrow T_{\e + \de}$ are identities for all $\e, \de\in [0,\infty)$. This implies that $T_0 = Id_\Cat$ and $T_\e T_\de = T_{\e + \de}$.\footnote{The implication does not go the other way. There are non-strict flows such that $T_0 = Id_\Cat$ and $T_\e T_\de = T_{\e + \de}$.} A flow is \emph{essentially strict} if the natural transformations $T_\e T_\de \Rightarrow T_{\e + \de}$ and $Id_\Cat \Rightarrow T_0$ are natural isomorphisms.
\end{definition}

\begin{definition}
Two objects $A,B\in \mathcal{C}$ are weakly $\e$-interleaved if there are maps $\alpha: A \to T_\e B$ and $\beta: B \to T_\e A$ and a commutative diagram

\begin{center}
\begin{tikzpicture}
\matrix (m) [matrix of math nodes,row sep=3em,column sep=4em,minimum width=2em]
{
T_0 A & A & B & T_0 B \\
 & T_\e A & T_\e B & \\
T_{2\e} A & T_\e T_\e A & T_\e T_\e B & T_{2\e} B \\ 
};
\path[-stealth]
    (m-1-1) edge node [left] { } (m-3-1)
    (m-1-2) edge node [left] { } (m-1-1)
    		edge node[xshift=-15pt, yshift=5pt] [left] {$\alpha$ } (m-2-3)	
    (m-1-3) edge node [left] { } (m-1-4)
    		edge[shorten >=1.5cm,-] node[xshift=15pt, yshift=5pt] [right] {$\beta$ } (m-2-2)
    		edge[shorten <=1.5cm] node[xshift=15pt, yshift=5pt] [right] {$\beta$ } (m-2-2)
    (m-1-4) edge node [left] { } (m-3-4)
    (m-2-2) edge node[xshift=-15pt, yshift=5pt] [left] {$T_\e \alpha$ }(m-3-3)
    (m-2-3) edge[shorten >=1.5cm,-] node[xshift=15pt, yshift=5pt] [right] {$T_\e \beta$ } (m-3-2)
    		edge[shorten <=1.5cm] node[xshift=15pt, yshift=5pt] [right] { } (m-3-2)
    (m-3-2) edge node [above] { } (m-3-1)
    (m-3-3) edge node [right] { } (m-3-4);
           
\end{tikzpicture}
\end{center}

and we call such a diagram a weak $\e$-interleaving of $A$ and $B$.
\end{definition}

\begin{definition}
The interleaving distance on a category with a flow $(\Cat, T_\e)$ is
$$d_{(\Cat, T_\e)}(A,B) = \inf \{\infty\} \cup \{\e\, \, | A\text{ and } B\text{ are weakly }\e\text{-interleaved}\}$$

When the context is clear, we will drop the subscript and simply write $d(A,B)$.
\end{definition}

\begin{remark}
In the past, most contexts with an interleaving distance used standard $\e$-interleavings (not weak ones). In fact, many contexts rewritten in this framework involve a category with a \emph{strict} flow. For categories with strict flows, weak $\e$-interleavings and the usual $\e$-interleavings are the same, so the interleaving distance is the usual one. 
\end{remark}
%As far as I can tell, Munch et al. care about general flows because they need them to generalize GPMs.

\begin{theorem} The interleaving distance has several desirable properties:

\begin{itemize}
\item $d(A,B) = 0$ if $A$ and $B$ are isomorphic.
\item $d(A,B) = d(B,A)$
\item $d(A,C) \leq d(A,B) + d(B,C)$
\end{itemize}

\end{theorem}

\begin{remark}
The interleaving distance is not quite a metric. Distances can be infinite, and distances can be zero even if the two objects are not isomorphic. 
\end{remark}

\begin{definition}
A strict flow-equivariant functor $H:(\Cat, T_\e) \to (\mathcal{D}, S_\e)$ between categories with a flow is an ordinary functor $H:\Cat \to \mathcal{D}$ so that $HT_\e = S_\e H$ and the following diagrams commute

\begin{center}
\begin{tikzpicture}
\matrix (m) [matrix of math nodes,row sep=3em,column sep=4em,minimum width=2em]
{
H & H\\
HT_0 & S_0 H\\
};
\path[-stealth]
   
    (m-1-1) edge[double, double distance=2pt, -] node [right] { } (m-1-2)		
    		edge node [left] { } (m-2-1)
    (m-1-2) edge node [right] { } (m-2-2)
	(m-2-1) edge[double, double distance=2pt, -] node [right] { } (m-2-2);
            
\end{tikzpicture}
\hspace{0.5in}
\begin{tikzpicture}
\matrix (m) [matrix of math nodes,row sep=3em,column sep=4em,minimum width=2em]
{
HT_\de & S_\de H  \\
HT_\e & S_\e H\\
};
\path[-stealth]
   
    (m-1-1) edge[double, double distance=2pt, -] node [right] { } (m-1-2)		
    		edge node [left] { } (m-2-1)
    (m-1-2) edge node [right] { } (m-2-2)
	(m-2-1) edge[double, double distance=2pt, -] node [right] { } (m-2-2);
            
\end{tikzpicture}
\hspace{0.5in}
\begin{tikzpicture}
\matrix (m) [matrix of math nodes,row sep=3em,column sep=4em,minimum width=2em]
{
HT_\e T_\de & S_\e S_\de H\\
HT_{\e + \de} & S_{\e + \de} H\\
};
\path[-stealth]
   
     (m-1-1) edge[double, double distance=2pt, -] node [right] { } (m-1-2)		
    		edge node [left] { } (m-2-1)
    (m-1-2) edge node [right] { } (m-2-2)
	(m-2-1) edge[double, double distance=2pt, -] node [right] { } (m-2-2);
            
\end{tikzpicture}
\end{center}
\end{definition}

Note that our terminology differs slightly from \cite{MunchCatsFlow}.\footnote{We do this because ``flow-equivariant'' can be dualized to ``coflow-equivariant'' when we talk about categories with coflows. The alternative was to call the analogous functors for coflows ``$[0,\infty)^{op}$-equivariant'', which seemed worse.} Also, while this definition is strong enough for the purposes of this paper, \cite{MunchCatsFlow} defines a more general kind of functor and proves a Stability theorem at this level of generality. We will content ourselves with using the more specific version stated here.

\begin{theorem} (Soft Stability Theorem) Let $H$ be a strict flow-equivariant functor between $(\Cat, T_\e)$ and $(\mathcal{D}, S_\e)$. Then $H$ is 1-Lipschitz with respect to the categories' interleaving distances.
\end{theorem} 

\subsection{Coflows}

It is not surprising that a categorical structure like a category with a flow would have a dual structure. We call this dual structure a category with a coflow. \\

There is a duality between categories with a flow and categories with a coflow. The following can be thought of as a definition, lemma, or remark, depending on your point of view:

\begin{definition} $(\Cat, T_\e)$ is a category with a coflow if and only if $(\Cat^{op}, T_\e^{op})$ is a category with a flow.
\end{definition}

Some might prefer the language of actegories. Then if a category with a flow is a $[0,\infty)$-actegory, a category with a coflow is a $[0,\infty)^{op}$-actegory.\\

All of the previous section can be dualized: there is a notion of weak $\e$-interleavings, of an interleaving distance, and of strict coflow-equivariant functors for categories with coflows. Likewise, the Triangle Inequality for the interleaving distance and the Soft Stability Theorem are still true for categories with a coflow.\\

In general, categories with coflows come up in cohomological contexts, while categories with flows come up in homological contexts. For example, persistence modules model persistent homology and form a category with a flow. Persistent cohomology lives in the opposite category, which is a category with a coflow. The category of sheaves and the category of derived sheaves are two more examples of categories with interleaving distances which come from a coflow.

\section{Cauchy Sequences in Categories with Interleavings}\label{sec:Main}

\subsection{Definition and Basic Results}
We will use the following lemma repeatedly and implicitly in this section.

\begin{lemma} The natural transformation $T_\de \Rightarrow T_{\e + \de}$ factors through $T_\de \Rightarrow T_0T_\de \Rightarrow T_\e T_\de \Rightarrow T_{\e + \de}$.
\end{lemma}
\begin{proof}
Combining two diagrams from section 2, we get the commutative diagram 

\begin{center}
\begin{tikzpicture}
\matrix (m) [matrix of math nodes,row sep=3em,column sep=4em,minimum width=2em]
{
 & T_\de  &  & T_{\e + \de} \\
T_\de & T_0 T_\de & & T_\e T_\de \\
};
\path[-stealth]

    (m-1-2) edge node [above] {$ T_{\de \leq \e+\de}$ } (m-1-4)
    (m-2-1) edge node [below] {$u\, I_{T_\de}$} (m-2-2)	
    		edge[double, double distance=2pt, -]  node [above] { } (m-1-2)
    (m-2-2) edge node [right] {$\mu_{0,\de}$ } (m-1-2)
    		edge node [below] {$T_{0\leq \e} T_{\de\leq \de }$} (m-2-4)
    (m-2-4) edge node [right] {$ \mu_{\e, \de}$ } (m-1-4);

\end{tikzpicture}
\end{center}
\end{proof}

\begin{definition} Let $\Cat$ be a category with a flow. Define a sequence of objects $\{A_n\}$ to be \emph{Cauchy} if for every $\e > 0$, there is an $N_\e$ so that for all $n,m \geq N_\e$, $d(A_n, A_m) < \e$.\\

$A$ is a \emph{metric limit} of the sequence $\{A_n\}$ if for every $\e > 0$, there is an $N_\e$ so that for all $n > N_\e$, $d(A_n, A) < \e$. If every Cauchy sequence $\{A_n\}$ in $\Cat$ has a metric limit, we say $\Cat$ is \emph{metrically complete}.
\end{definition}

\begin{remark}This definition is probably not surprising, but there is a small reason to include it. Cauchy sequences are usually defined for sets with metrics (and quasi-metrics, and extended quasi-metrics); that is to say, for sets. But in this context, we are not guaranteed that the collection of objects of $\Cat$ forms a proper set. Nonetheless, this definition still makes sense.  
\end{remark} 

\begin{remark} Metric limits of a Cauchy sequence of objects of $\Cat$ are not necessarily unique, even up to isomorphism. However, we have the following lemma
\end{remark}

\begin{lemma} Let $A$ be a metric limit of a Cauchy sequence $\{A_n\}$. Then $B$ is also a metric limit of the same Cauchy sequence if and only if $d(A,B) = 0$. 
\end{lemma}
\begin{proof}
Say $d(A,B) = 0$. Then for all $\e$, there is an $N_\e$ so that for all $n > N_\e$, $d(A_n, A) < \e$. Then $d(A_n, B) \leq d(A_n, A) + d(A,B) = d(A_n, A) < \e$, so $B$ is a metric limit of $\{A_n\}$. \\

Say $B$ is another metric limit of $\{A_n\}$. There is some $N_\e$ so that $d(A_n, A) < \e$ and $d(A_n, B) < \e$ for all $n \geq N_\e$. Therefore, $d(A, B) \leq d(A_n, A) + d(A_n, B) \leq 2\e$. Since $\e$ was arbitrary, $d(A,B) = 0$. 
\end{proof} 

Let $\{\tilde{A}_n\}$ be a Cauchy sequence of elements in $\Cat$. Let $\e_k = 2^{-k}$. Let $N_k$ be a sequence of natural numbers so that $N_k < N_{k+1}$ and for all $n,m \geq N_k$, $d(\tilde{A}_n, \tilde{A}_m) < \e_k/2 = \e_{k+1}$. Set $A_k = \tilde{A}_{N_k}$.

\begin{lemma}\label{lemma:subsequence} $\{A_k\}$ is a Cauchy sequence, and if $A$ is a metric limit of $\{\tilde{A}_k\}$ if and only if $A$ is a metric limit of $\{A_k\}$. 
\end{lemma}

The proof of this lemma is the same as the proof that a subsequence of a Cauchy sequence in a metric space is Cauchy with the same limit, which can be found in any undergraduate analysis text.

\subsection{Categorical conditions implying metric completeness}

Take a Cauchy sequence $\{A_k\}$ so that $d(A_k, A_n) < \e_{k+1} = 2^{-(k+1)}$ for all $n > k$. For each $k \geq 1$, choose an $\e_{k+1}$-interleaving between $A_k$ and $A_{k+1}$. This gives us a map $$T_{\e_{k+1}}A_{k+1} \to T_{\e_{k+1}}T_{\e_{k+1}}A_{k+1} \to T_{2\e_{k+1}}A_k = T_{\e_k}A_k$$

Patching these together, we get the diagram

\begin{equation*}
\cdots \to T_{\e_{k+1}}A_{k+1} \to T_{\e_k}A_k \to \cdots \to T_{\e_2}A_2 \to T_{\e_1}A_1\label{eq:MainSequence}
\end{equation*}

If the categorical limit exists, set $A = \lim (\cdots \to T_{\e_{k+1}}A_{k+1} \to T_{\e_k}A_k \to \cdots \to T_{\e_2}A_2 \to T_{\e_1}A_1)$. Denote the limit maps $\phi_k: A \to  A_k$. 

\begin{theorem}\label{thm:CatLimitEqualsMetricLimit} Let $(\Cat, T_\e)$ be a category with a flow where $T_\e$ preserves categorical limits. Let $\{A_k\}$ be a sequence of objects so that $d(A_k, A_n) < \e_{k+1}$ for all $k$ and $n > k$. For all $k$, pick a weak $\e_{k+1}$-interleaving between $A_k$ and $A_{k+1}$, so you get a map $T_{\e_{k+1}}A_{k+1} \to T_{\e_k} A_k$. Let $A$ be a categorical limit of the diagram 
$$\cdots \to T_{\e_{k+1}}A_{k+1} \to T_{\e_k} A_k \to \cdots \to T_{\e_2}A_2 \to T_{\e_1} A_1$$

Then $A$ is a metric limit of $\{A_k\}$. 
\end{theorem}
\begin{proof}We will show that $d(A_k, A) < \e_k$ by exhibiting a weak $\e_k$-interleaving between $A_k$ and $A$. This will give us that $A$ is a metric limit of $\{A_k\}$. We will actually give a strict $\e_k$-interleaving.\\

First, consider the following commutative diagram

\begin{center}
\begin{tikzpicture}[scale = 0.94, every node/.style={scale=0.94}]
\matrix (m) [matrix of math nodes,row sep=3em,column sep=2.5em,minimum width=2em]
{

A_k & T_{\e_k}A_k & T_0T_{\e_k}A_k & T_{\e_k}T_{\e_k} A_k\\
T_{\e_{k+1}}A_{k+1} & T_{\e_{k+1}}T_{\e_{k+1}} A_{k+1}& T_{\e_{k+1}}T_{\e_{k+1}}A_{k+1}&T_{\e_k}T_{\e_{k+1}} A_{k+1}\\
T_{\e_{k+1}}T_{\e_{k+2}}A_{k+2} & T_{\e_{k+1}}T_{\e_{k+2}}T_{\e_{k+2}} A_{k+2}& T_{\e_{k+1}+\e_{k+2}}T_{\e_{k+2}}A_{k+2} &T_{\e_k}T_{\e_{k+2}} A_{k+2}\\
T_{\e_{k+1}}T_{\e_{k+2}}T_{\e_{k+3}}A_{k+3} & T_{\e_{k+1}}T_{\e_{k+2}}T_{\e_{k+3}}T_{\e_{k+3}} A_{k+3}& T_{\e_{k+1}+\e_{k+2}+\e_{k+3}} T_{\e_{k+3}} A_{k+3}&T_{\e_k}T_{\e_{k+3}} A_{k+3}\\
};
\path[-stealth]
    (m-1-1) edge node [above] { } (m-1-2)
		    edge node [above] { } (m-2-1)
	(m-1-2) edge node [above] { } (m-1-3)
	(m-1-3) edge node [above] { } (m-1-4)
	(m-2-1) edge node [above] { } (m-2-2)
			edge node [above] { } (m-3-1)
			edge node [above] { } (m-1-2)
	(m-2-2) edge node [above] { } (m-2-3)
	(m-2-3) edge node [above] { } (m-2-4)
	(m-2-4) edge node [above] { } (m-1-4)
	(m-3-1) edge node [above] { } (m-3-2)	
			edge node [above] { } (m-2-2)
			edge node [above] { } (m-4-1)
	(m-3-2) edge node [above] { } (m-3-3)	    
	(m-3-3) edge node [above] { } (m-3-4)
	(m-3-4) edge node [above] { } (m-2-4)
	(m-4-1) edge node [above] { } (m-4-2)	
			edge node [above] { } (m-3-2)
	(m-4-2) edge node [above] { } (m-4-3)
	(m-4-3) edge node [above] { } (m-4-4)
	(m-4-4) edge node [above] { } (m-3-4);
            
\end{tikzpicture}
\end{center}

The first triangle we get from the weak $\e_{k+1}$-interleaving between $A_k$ and $A_{k+1}$, the second triangle we get from applying $T_{\e_{k+1}}$ to the weak $\e_{k+2}$-interleaving between $A_{k+1}$ and $A_{k+2}$, and so on. The trapezoids we get from the natural transformation $T_{\e_{k+1}}T_{\e_{k+2}} ... T_{\e_{k+m}} \Rightarrow T_{\e_k}$.\\

For $m \geq k$, denote by $\psi_m: A_k \to T_{\e_k}T_{\e_m}A_m$ the map which we get from this diagram. For $m < k$, denote by $\psi_m: A_k \to T_{\e_k}T_{\e_m}A_m$ the map which makes the rest of the following diagram commute:

\begin{center}
\begin{tikzpicture}
\matrix (m) [matrix of math nodes,row sep=3em,column sep=4em,minimum width=2em]
{

& & A_k &   \\
\cdots & T_{\e_k}T_{\e_{m+1}} A_{m+1}  & T_{\e_k}T_{\e_m}A_m & \cdots \\
};
\path[-stealth]
    (m-1-3) edge node[xshift=-10pt, yshift=0pt] [above] {$\psi_{m+1}$ } (m-2-2)
    		edge node [left] {$\psi_m$ } (m-2-3)
    (m-2-1) edge node [below] { }(m-2-2)
    (m-2-2) edge node [right] { } (m-2-3)
    (m-2-3) edge node [right] { } (m-2-4);
            
\end{tikzpicture}
\end{center}

The collection of maps $\psi_m$ induces a map $\psi: A_k \rightarrow T_{\e_k}A$, because $T_{\e_k}A$ is the limit of that bottom diagram.\\

Now we have the following maps:
\begin{itemize}
\item $\phi_k: A \to T_{\e_k}A_{\e_k}$. This comes from the definition of $A$ as a limit. 
\item $T_{\e_k}\phi_k: T_{\e_k}A \to T_{\e_k}T_{\e_k} A_k$. This comes from applying $T_{\e_k}$ to $\phi_i$. By hypothesis, it is also the limit map of $T_{\e_k}A = \lim\left(\cdots \to T_{\e_k}T_{\e_2}A_2 \to T_{\e_k}T_{\e_1}A_1\right)$.
\item $\psi: A_k \to T_{\e_k}A$, which we just defined.
\item $T_{\e_k}\psi: T_{\e_k}A_k \to T_{\e_k}T_{\e_k}A $. This comes from applying $T_{\e_k}$ to $\psi$. 
\end{itemize}

We must check that these pentagons commute:

\begin{center}
\begin{tikzpicture}
\matrix (m) [matrix of math nodes,row sep=3em,column sep=4em,minimum width=2em]
{
T_0 A_k   &A_k  &  \\
& &T_{\e_k} A \\
T_{2\e_k}A_k & T_{\e_k}T_{\e_k} A_k & \\
};
\path[-stealth]
    (m-1-1) edge node [left] { } (m-3-1)
    (m-1-2)	edge node[xshift=0pt, yshift=5pt] [right] {$\psi$ } (m-2-3)	
    		edge node [left] { } (m-1-1)
    (m-2-3) edge node[xshift=0pt, yshift=-7pt] [right] {$T_{\e_k}\phi_k$ } (m-3-2)
    (m-3-2) edge node [right] { } (m-3-1);
\end{tikzpicture}
\hspace{0.5in}
\begin{tikzpicture}
\matrix (m) [matrix of math nodes,row sep=3em,column sep=4em,minimum width=2em]
{
& A & T_0 A\\
T_{\e_k}A_k & &  \\
& T_{\e_k}T_{\e_k} A_k & T_{2\e_k} A_k\\
};
\path[-stealth]
   
    (m-1-2) edge node[xshift=0pt, yshift=7pt] [left] {$T_{\e_k}\psi$ } (m-2-1)	
    		edge node [right] { } (m-1-3)		
    (m-2-1) edge node[xshift=0pt, yshift=-5pt] [left] {$\phi_k$} (m-3-2)
    (m-3-2) edge node [left] { } (m-3-3)
    (m-1-3) edge node [left] { } (m-3-3);
            
\end{tikzpicture}
\end{center}

We will do this by showing the stronger statement that these triangles commute:

\begin{center}
\begin{tikzpicture}
\matrix (m) [matrix of math nodes,row sep=3em,column sep=4em,minimum width=2em]
{
A_k   &  & T_{\e_k}T_{\e_k}A_k \\
& T_{\e_k}A&\\
};
\path[-stealth]
    (m-1-1) edge node [left] { } (m-1-3)
    		edge node[xshift=-2pt, yshift=-3pt] [left] {$\psi$ } (m-2-2)		
    (m-2-2) edge node[xshift=5pt, yshift=-3pt] [right] {$T_{\e_k}\phi_k$ } (m-1-3);
\end{tikzpicture}
\hspace{0.5in}
\begin{tikzpicture}
\matrix (m) [matrix of math nodes,row sep=3em,column sep=4em,minimum width=2em]
{
 & T_{\e_k}A_k  &  \\
A & & T_{\e_k}T_{\e_k}A\\
};
\path[-stealth]
   
    (m-1-2) edge node[xshift=-3pt, yshift=6pt] [right] {$T_{\e_k}\psi$ } (m-2-3)		
    (m-2-1) edge node[xshift=3pt, yshift=6pt] [left] {$\phi_k$} (m-1-2)
		    edge node [left] { } (m-2-3);
            
\end{tikzpicture}
\end{center}

The first commutes by the definition of $\psi$. The second triangle is harder. Notice we have the commutative diagram

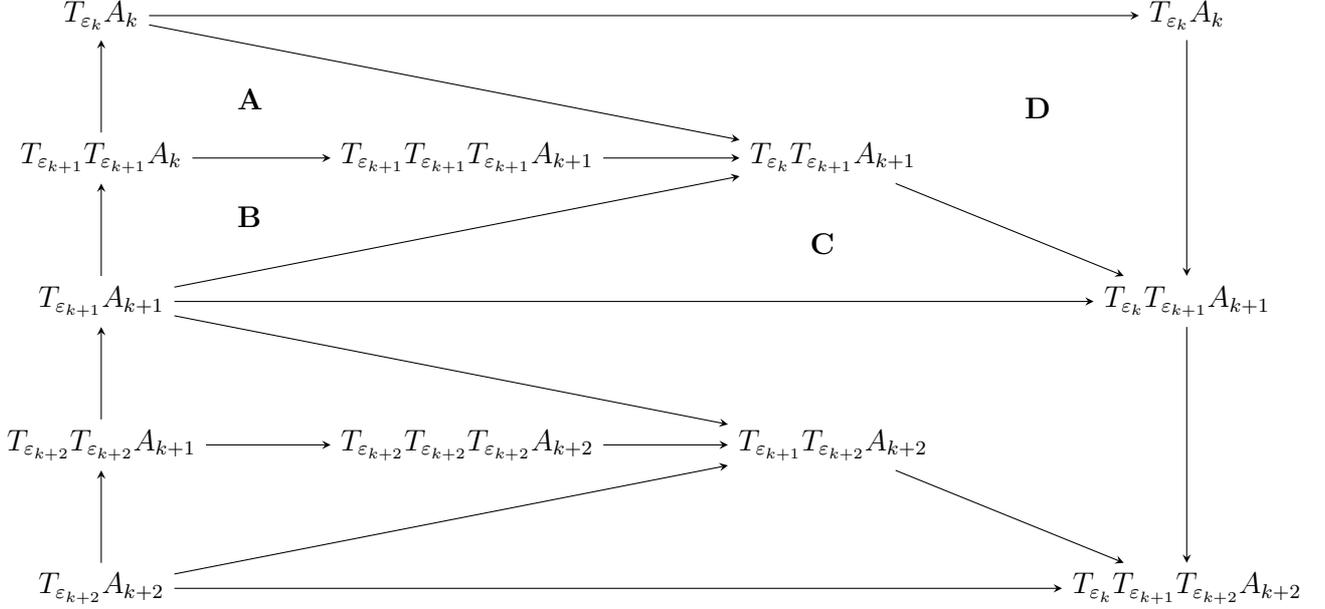
\begin{center}
\begin{tikzpicture}[scale = 0.94, every node/.style={scale=0.94}]
\matrix (m) [matrix of math nodes,row sep=3em,column sep=4em,minimum width=2em]
{
T_{\e_k}A_k & & & T_{\e_k}A_k \\
T_{\e_{k+1}}T_{\e_{k+1}}A_k & T_{\e_{k+1}}T_{\e_{k+1}}T_{\e_{k+1}}A_{k+1} & T_{\e_{k}}T_{\e_{k+1}}A_{k+1} &\\
T_{\e_{k+1}}A_{k+1} &   &   &T_{\e_k}T_{\e_{k+1}}A_{k+1}\\
T_{\e_{k+2}}T_{\e_{k+2}}A_{k+1} & T_{\e_{k+2}}T_{\e_{k+2}}T_{\e_{k+2}}A_{k+2}&  T_{\e_{k+1}}T_{\e_{k+2}}A_{k+2} &\\
T_{\e_{k+2}}A_{k+2} &     &  &T_{\e_k}T_{\e_{k+1}}T_{\e_{k+2}}A_{k+2}\\
};
\path[-stealth]

	(m-1-1) edge node [right] {  } (m-1-4)	
			edge node [left] { } (m-2-3)
	(m-1-4) edge node[xshift=-50pt, yshift=20pt] [left] {$\textbf{D}$ } (m-3-4)	
    (m-2-1) edge node[xshift=50pt, yshift=-5pt] [right] {$\textbf{A}$ } (m-1-1)
		    edge node [left] { } (m-2-2)
	(m-2-2) edge node [left] { } (m-2-3)
	(m-3-1) edge node[xshift=50pt, yshift=5pt] [right] {$\textbf{B}$ } (m-2-1)
		    edge node [left] { } (m-2-3)
		    edge node[xshift=75pt, yshift=15pt] [above] { $\textbf{C}$} (m-3-4)
		    edge node [left] { } (m-4-3)
    (m-2-3) edge node [left] { } (m-3-4)
    (m-3-4) edge node [left] { } (m-5-4)
    (m-4-1) edge node [left] { } (m-3-1)
    		edge node [left] { } (m-4-2)
    (m-5-1) edge node [left] { } (m-4-1)
    		edge node [left] { } (m-4-3)
    		edge node [left] { } (m-5-4)	
    (m-4-2) edge node [left] { } (m-4-3)
    (m-4-3) edge node [left] { } (m-5-4)   ;

;
\end{tikzpicture}
\end{center}

The $\textbf{A}$ triangles come from applying the natural transformations $T_{\e_{k+m+1}}T_{\e_{k+m+1}} \Rightarrow T_{\e_{k+m}}$ to the interleaving maps $A_{k+m} \to T_{\e_{k+m+1}}A_{k+m+1}$. The $\textbf{B}$ triangles come from the  weak $\e_{k+m+1}$-interleaving between $A_{k+m}$ and $A_{k+m+1}$. The horizontal maps are defined so that the $\textbf{C}$ triangles commute. Lastly, we know the $\mathbf{D}$ diagrams commute by applying the natural transformations $Id \Rightarrow T_{\e_k}T_{\e_{k+1}}...T_{\e_{k+m}}$ to the maps $A_{k+m} \to T_{\e_{k+m+1}}A_{k+m+1}$.\\

We can concatenate this diagram with $T_{\e_k}$ applied to an earlier diagram to get

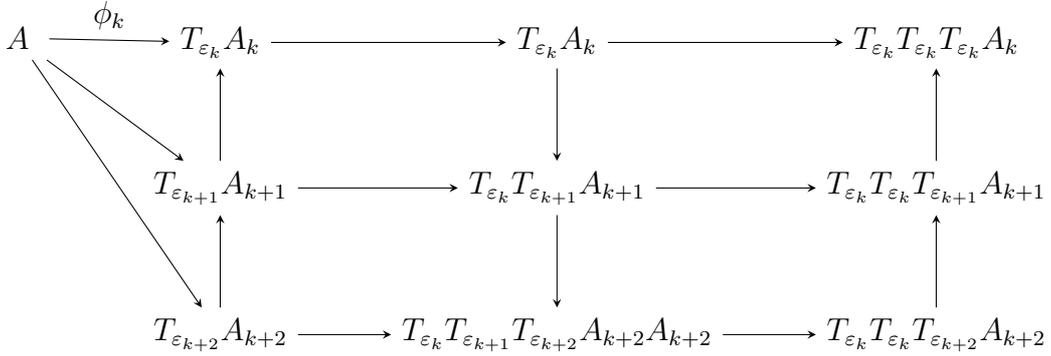
\begin{center}
\begin{tikzpicture}
\matrix (m) [matrix of math nodes,row sep=3em,column sep=3em,minimum width=2em]
{

A & T_{\e_k}A_k & T_{\e_k}A_k & T_{\e_k}T_{\e_k}T_{\e_k} A_k\\
& T_{\e_{k+1}} A_{k+1} & T_{\e_k}T_{\e_{k+1}}A_{k+1} & T_{\e_k}T_{\e_k}T_{\e_{k+1}} A_{k+1}\\
& T_{\e_{k+2}} A_{k+2}& T_{\e_k}T_{\e_{k+1}}T_{\e_{k+2}}A_{k+2} A_{k+2}& T_{\e_k}T_{\e_k}T_{\e_{k+2}} A_{k+2}\\
};
\path[-stealth]
	(m-1-1) edge node [above] {$\phi_k$} (m-1-2)
			edge node [above] { } (m-2-2)
			edge node [above] { } (m-3-2)
	(m-1-2) edge node [above] { } (m-1-3)
	(m-2-2) edge node [above] { } (m-1-2)
		    edge node [above] { } (m-2-3)	 
	(m-3-2) edge node [above] { } (m-2-2)
		    edge node [above] { } (m-3-3)	 	       
    (m-1-3) edge node [above] { } (m-1-4)
		    edge node [above] { } (m-2-3)
	(m-2-3) edge node [above] { } (m-2-4)
			edge node [above] { } (m-3-3)
	(m-2-4) edge node [above] { } (m-1-4)
	(m-3-3) edge node [above] { } (m-3-4)	
	(m-3-4) edge node [above] { } (m-2-4);
            
\end{tikzpicture}
\end{center}

This diagram tells us $T_{\e_k}T_{\e_k}T_{\e_m}\phi_m = T_{\e_k}\psi_m\circ \phi_k$. These maps induce a map from $A$ to $T_{\e_k}T_{\e_k}A$. By continuity of $T_{\e_k}$, $T_{\e_k}T_{\e_k}T_{\e_m}\phi_m$ induces the map $T_{\e_k}T_{\e_k}T_{\e_m}\phi$, and $T_{\e_k}\psi_m\circ \phi_k$ induces the map $T_{\e_k}\psi\circ \phi_k$. Therefore, our second triangle commutes.\\

This shows there is an $\e_k$-interleaving between $A_k$ and $A$, and therefore $d(A_k, A) < \e_k$. Thus, $A$ is a metric limit of $\{A_k\}$. 
\end{proof}

\begin{theorem}\label{thm:CategorialConditionsForCompleteness}A category with a flow $(\Cat, T_\e)$ is metrically complete if
\begin{itemize}
\item $\Cat$ contains all limits of the form $\cdots\to\bullet \to \bullet \to \bullet$, and
\item for all $\e > 0$, $T_\e$ preserves categorical limits.
\end{itemize}

Dually, a category with a coflow $(\Cat, T_\e)$ is metrically complete if 
\begin{itemize}
\item $\Cat$ contains all colimits of the form $\bullet \to \bullet \to \bullet\to \cdots $, and
\item for all $\e > 0$, $T_\e$ preserves (categorical) colimits.
\end{itemize}
\end{theorem}
\begin{proof}
We show the proof for the category with a flow; dualize this to get a proof for a category with a coflow. \\

Let $\{\tilde{A}_k\}$ be a Cauchy sequence in $(\Cat, T_\e)$. As described before, there is a subsequence $\{A_k\}$ so that $d(A_k, A_n) < \e_{k+1}$ for all $n > k$. Then we can create a diagram 

$$\cdots \to T_{\e_{k+1}}A_{k+1} \to T_{\e_k} A_k \to \cdots \to T_{\e_2}A_2 \to T_{\e_1} A_1$$

$\Cat$ has categorical limits of all diagrams of this form; denote this limit $A$. Then by Theorem \ref{thm:CatLimitEqualsMetricLimit}, $A$ is a metric limit of $\{A_k\}$, and by Lemma \ref{lemma:subsequence}, $A$ is also a metric limit of $\{\tilde{A}_k\}$. \\

Thus, every Cauchy sequence has a limit, and $(\Cat, T_\e)$ is metrically complete. 
\end{proof}

\section{Metric Completions}\label{sec:Completions}
In this section, we will work with categories with a coflow for convenience's sake. The results we reference from outside sources are more familiar in the coflow case. The metric completion of a category with a flow can be found by dualizing and completing the resulting category with a coflow.\\

\begin{lemma}\label{lemma:definitionOfCompletion}Let $(\Cat, T_\e)$ and $(\mathcal{D}, S_\e)$ be categories with coflows, and $(\Cat, T_\e)$ be a strict coflow-equivariant, full subcategory of $(\mathcal{D}, S_\e)$. Let $T_\e$ be a strict flow, $\mathcal{D}$ be cocomplete, and $S_\e$ preserve colimits. Then there is a full subcategory $\mathcal{Z}$ of $\mathcal{D}$ so that
\begin{enumerate}
\item $\Cat$ is a full subcategory of $\mathcal{Z}$.
\item $\Cat$ is metrically dense in $\mathcal{Z}$; i.e. for every object $Z\in \mathcal{Z}$ and real number $\e > 0$, there is an object $A \in \Cat$ so that $d(A,Z) < \e$. 
\item $S_\e$ preserves $\mathcal{Z}$. 
\item $(\mathcal{Z}, S_\e)$ is metrically complete. 
\end{enumerate}
\end{lemma}

\begin{proof}
Let $\{\tilde{A}_k\}$ be a Cauchy sequence in $\mathcal{C}$, and let $\{A_k\}$ be a subsequence so that $d(A_n,A_m) < \e_{k+1}$ for all $n,m \geq k$, as in section \ref{sec:Main}. Then we get as before a diagram 
\begin{equation*}
T_{\e_1}A_1 \to T_{\e_2}A_2 \to \cdots \to T_{\e_k}A_k \to T_{\e_{k+1}}A_{k+1}\to\cdots 
\end{equation*}

We can consider this a diagram in $\mathcal{D}$. It has a colimit $A$ in $\mathcal{D}$, because that category has all small colimits.\\

Let the objects of $\mathcal{Z}$ be the objects of $\mathcal{C}$ along with any object of $\mathcal{D}$ which is a colimit $A$ of a diagram of the form given above. Then $\mathcal{Z}$ is defined to be the full subcategory of $\mathcal{D}$ with that collection of objects.\\

We will now show that $\mathcal{Z}$ satisfies the properties of this corollary. (1) and (2) follow directly from the definition. To show (3), first note that if $Z$ is an object of $\mathcal{Z}$, $Z$ is the colimit of some diagram 
\begin{equation*}
T_{\e_1}A_1 \to T_{\e_2}A_2 \to \cdots \to T_{\e_k}A_k \to T_{\e_{k+1}}A_{k+1}\to\cdots 
\end{equation*}
of the type above. Then $S_\e Z$ is the colimit of the diagram 

\begin{equation*}
S_\e T_{\e_1}A_1 \to S_\e T_{\e_2}A_2 \to \cdots \to S_\e T_{\e_k}A_k \to S_\e T_{\e_{k+1}}A_{k+1}\to\cdots 
\end{equation*}

Because the embedding preserves the flow, this diagram is equivalent to 

\begin{equation*}
T_\e T_{\e_1}A_1 \to T_\e T_{\e_2}A_2 \to \cdots \to T_\e T_{\e_k}A_k \to T_\e T_{\e_{k+1}}A_{k+1}\to\cdots 
\end{equation*}

which by strictness of $T_\e$ is the same as 

\begin{equation*}
T_{\e_1}\left(T_\e A_1\right) \to T_{\e_2}\left(T_\e A_2\right) \to \cdots \to T_{\e_k}\left(T_\e A_k\right) \to T_{\e_{k+1}}\left( T_\e A_{k+1}\right)\to\cdots 
\end{equation*}

Therefore, $S_\e Z$ is also an object of $\mathcal{Z}$, and we've shown (3). \\

Lastly, we'll show (4). Take a Cauchy sequence $\{Z_k\}$ of objects of $\mathcal{Z}$. Then by (3), for each $Z_k$ there is an object $A_k$ in $\mathcal{C}$ so that $\ds d(A_k, Z_k) < 2^{-k}$. Then $\{A_k\}$ is also a Cauchy sequence with the same set of metric limits as $\{Z_k\}$. Because $\{A_k\}$ consists of objects of $\mathcal{C}$, it has a metric limit in $\mathcal{Z}$. Therefore, $\{Z_k\}$ also has a metric limit in $\mathcal{Z}$.  
\end{proof}

\begin{definition}\label{def:Closure} We denote the largest subcategory satisfying the properties of Lemma \ref{lemma:definitionOfCompletion}  as the \emph{closure of $(\Cat, T_\e)$ in $(\mathcal{D},S_\e)$}.  
\end{definition}
\begin{proof}This is one of those definitions that requires proof. Let $A$ be an object of $\mathcal{X}$ if and only if $A$ is an object of a subcategory $\mathcal{Z}$ of $\mathcal{D}$ satisfying the conditions of Lemma \ref{lemma:definitionOfCompletion}; Let $\mathcal{X}$ be the full subcategory with this class of objects. It is clear that any category that satisfies the conditions of Lemma \ref{lemma:definitionOfCompletion} must be a subcategory of $\mathcal{X}$. We will show $\mathcal{X}$ also satisfies those conditions. \\

Conditions 1, 2, and 3 are easy to verify. Condition 4 can be shown as follows. Take a Cauchy sequence $\{A_k\} \subset \mathcal{X}$. By condition 2, for each $k$ there is an object $B_k \in \Cat$ so that $d(A_k, B_k) \leq \frac{1}{k}$. Now $\{B_k\}$ is a Cauchy sequence of objects in $\Cat$, and an object $A\in \mathcal{D}$ is a metric limit of $\{A_k\}$ if and only if it is a metric limit of $\{B_k\}$. Since Lemma \ref{lemma:definitionOfCompletion} says there is a subcategory $\mathcal{Z}$ which is metrically complete and contains each $B_k$, $\{B_k\}$ has a metric limit $A \in \mathcal{Z} \subseteq \mathcal{X}$. Thus, $\{A_k\}$ has a metric limit in $\mathcal{X}$, and $\mathcal{X}$ is metrically complete.   
\end{proof}

Let $\Cat$ be a locally small category (i.e. for every pair of objects $A,B \in \Cat$, $Hom(A,B)$ is a proper set). Let $h_{-}$ be the Yoneda functor which takes an object $A$ in $\Cat$ and sends it to a contravariant functor from $\Cat$ to $\mathbf{Set}$ like so:
$$h_- : A \mapsto h_A, \text{ where } h_A(B) = Hom(B,A)$$

The Yoneda Lemma tells us that the natural transformations between the functors $h_A$ and $h_B$ are in one-to-one correspondence with the morphisms between $A$ and $B$. In other words, $Nat(h_A, h_B) = Hom(A,B)$. \\

One interpretation of this is that $\mathcal{C}$ can be embedded as a full subcategory of $\left[\Cat^{op},\mathbf{Set}\right]$, the category of contravariant functors from $\Cat$ to $\mathbf{Set}$.\footnote{This is sometimes called the category of set-valued presheaves on $\Cat$.} This is called the Yoneda embedding.\footnote{ Information on the Yoneda embedding and on category theory more generally can be found in \cite{Riehl} and \cite{MacLane}.}

\begin{lemma}\label{lemma:YonedaFacts} Assume $(\Cat, T_\e)$ is a small category. Then
\begin{enumerate}
\item $\left[\Cat^{op},\mathbf{Set}\right]$ is locally small.
\item $\left[\Cat^{op},\mathbf{Set}\right]$ is complete and cocomplete, meaning it contains all small limits and small colimits. In particular, it contains all colimits of diagrams of the form $\bullet \to \bullet \to \bullet \to \cdots$.
%\item The Yoneda embedding $h_-: \Cat \to \left[\Cat^{op},\mathbf{Set}\right]$ is (categorically) continuous, meaning it preserves small limits.
\item For every $\e > 0$, there is a pointwise left Kan extension $L_\e  = \text{\emph{Lan}}_{h_-} h_-\circ T_\e$ 
\begin{center}
\begin{tikzpicture}
\matrix (m) [matrix of math nodes,row sep=3em,column sep=4em,minimum width=2em]
{
\Cat & \Cat & \left[\Cat^{op},\mathbf{Set}\right]\\
\left[\Cat^{op},\mathbf{Set}\right]\\
};
\path[-stealth]
    (m-1-1) edge node [above] {$T_\e$ } (m-1-2)
    		edge node [left] {$h_-$ } (m-2-1)
    (m-1-2) edge node [above] {$h_-$ } (m-1-3)	
    (m-2-1) edge[dashed] node [below] {$L_\e$ } (m-1-3);
            
\end{tikzpicture}
\end{center}
and we can choose $L_\e$ so that $L_\e\circ h_- = h_-\circ T_\e$. 
\item $L_\e$ preserves all small colimits. 
\item $L_\e$ has a right adjoint. 
\item $L_0 = Id_{[\Cat^{op}, \mathbf{Set}]}$.
\end{enumerate}
\end{lemma}

\begin{proof}
(1) follows directly from the Yoneda Lemma. Limits and colimits in functor categories are computed ``pointwise'', so (2) follows from the completion and cocompletion of $\mathbf{Set}$. The diagram above satisfies the conditions of Corollary 6.2.6 in \cite{Riehl}, which says the pointwise left Kan extension exists; that we can choose it to be a true extension is indicated in Corollary 4 in Section X.3 of \cite{MacLane}. (4) is a property of any pointwise left Kan extension along Yoneda; see Section 2.7 of \cite{KSCategoriesAndSheaves}. Further, because $[\Cat^{op}, \mathbf{Set}]$ is locally presentable, (5) follows from (4) by the Adjoint Functor Theorem. (6) is a famous result; another way this is said is that the Yoneda embedding is (categorically) codense (see Section X.6 in \cite{MacLane}).
\end{proof}

\begin{theorem} Assume $(\Cat, T_\e)$ is a category with a strict flow, and define the functors $L_\e$ as in Lemma \ref{lemma:YonedaFacts}. Then
\begin{enumerate}
\item $L_\e$ defines a coflow on $\left[\Cat^{op},\mathbf{Set}\right]$. 
\item $\left(\Cat, T_\e\right) \xrightarrow{h_-} \left(\left[\Cat^{op},\mathbf{Set}\right], L_\e\right)$ is a full, strict coflow-equivariant embedding of categories. 
\item $\left(\left[\Cat^{op},\mathbf{Set}\right], L_\e\right)$ is a metrically complete category.
\end{enumerate}
\end{theorem}

\begin{proof}
To show $L_\e$ defines a coflow, we must give natural transformations $L_\e \Rightarrow L_\de$ for $\de \leq \e$, $L_0 \Rightarrow Id$, and $L_{\e+\de} \Rightarrow L_\e L_\de$ so that all the diagrams from Section \ref{sec:Background} commute.\\

First, we can define the natural transformations $L_\e \Rightarrow L_\de$ by applying the universal property of a Left Kan extension to the diagram 

\begin{center}
\begin{tikzpicture}
\matrix (m) [matrix of math nodes,row sep=3em,column sep=4em,minimum width=2em]
{
\Cat & \Cat & \left[\Cat^{op},\mathbf{Set}\right]\\
\Cat & \Cat & \left[\Cat^{op},\mathbf{Set}\right]\\
\left[\Cat^{op},\mathbf{Set}\right]\\
};
\path[-stealth]
    (m-1-1) edge node [above] {$T_\e$ } (m-1-2)
    		edge[double, double distance=2pt, -] node [left] { } (m-2-1)
    (m-1-2) edge node [above] {$h_-$ } (m-1-3)
    		edge[double, double distance=2pt, -] node [above] { } (m-2-2)	
	(m-1-3) edge[double, double distance=2pt, -] node [above] { } (m-2-3)
    (m-2-1) edge node [above] {$T_\de$ } (m-2-2)
    		edge node [left] {$h_-$ } (m-3-1)
    (m-2-2) edge node [above] {$h_-$ } (m-2-3)
    (m-3-1) edge node [below] {$L_\de$ } (m-2-3);
            
\end{tikzpicture}
\end{center}

This also shows why $L_\e$ is functorial in the $\e$. Lemma \ref{lemma:YonedaFacts}.6 says $L_0 = Id$, so that natural transformation is the identity.\\

Notice also that $L_\e$, being a left adjoint, preserves left Kan extensions; see Theorem 1 of Section X.5 in \cite{MacLane}. In particular, 
\begin{align*}
L_\e \circ L_\de &\Leftarrow L_\e \circ \text{Lan}_{h_-} (h_- \circ T_\de)\\
&\Leftarrow \text{Lan}_{h_-} \left(L_\e \circ h_- \circ T_\de\right) \\
&= \text{Lan}_{h_-} \left(h_- \circ T_\e \circ T_\de \right)\\
&= \text{Lan}_{h_-} \left(h_- \circ T_{\e + \de}\right)\\
&\Leftarrow L_{\e + \de}
\end{align*}
where the arrows are the unique natural isomorphisms making the appropriate diagram commute. This is the natural transformation $L_{\e + \de} \Rightarrow L_\e L_\de$ we choose.\\

The only thing left to proving (1) is verifying that the four coflow relation diagrams commute. The ones involving $L_0$ are trivial because $L_0 =  Id$. The other two follow from the properties of a Kan extension.\\

Yoneda says $h_-$ is a full embedding of categories. Lemma \ref{lemma:YonedaFacts}.3 says this is a strict co-equivariant embedding of categories. This gives (2).\\

Lastly, Lemmas \ref{lemma:YonedaFacts}.2 and \ref{lemma:YonedaFacts}.4 tell us that $[\Cat^{op}, \mathbf{Set}]$ satisfies the conditions of Theorem \ref{thm:CategorialConditionsForCompleteness}. This gives (3). 
\end{proof}

\begin{corollary}\label{cor:ExistenceOfCompletion} For any category $(\Cat, T_\e)$ with a strict flow, there is a full subcategory $\overline{\Cat}$ of $\left[\Cat^{op},\mathbf{Set}\right]$ so that
\begin{enumerate}
\item $\Cat$ is a full subcategory of $\overline{\Cat}$.
\item $\Cat$ is metrically dense in $\overline{\Cat}$; i.e. for every object $Z\in \overline{\Cat}$ and real number $\e > 0$, there is an object $A \in \Cat$ so that $d(A,Z) < \e$. 
\item $L_\e$ preserves $\overline{\Cat}$. 
\item $(\overline{\Cat}, L_\e)$ is metrically complete.
\item $\overline{\Cat}$ is the largest subcategory of $\left[\Cat^{op},\mathbf{Set}\right]$ to satisfy properties 1-4. 
\end{enumerate}
\end{corollary}

\begin{proof}
The theorem is essentially checking the conditions of Lemma \ref{lemma:definitionOfCompletion}, of which this corollary is a direct application. We then just apply Definition \ref{def:Closure}.
\end{proof}

\begin{definition} We call $(\overline{\Cat}, L_\e)$ the \emph{metric completion of $(\Cat, T_\e)$}. 
\end{definition}

\section{Examples}\label{sec:Examples}

\subsection{$(\Q, \geq)$}
Consider the category $(\Q, \geq)$, with one object for every rational number, and $\ds \Hom(q,p) = \left\{\begin{array}{cl} \ast & \text{ if }q \geq p\\ \emptyset & \text{ otherwise}\end{array}\right.$, where $\ast$ is the set with one element and $\emptyset$ is the empty set. Define a coflow $T_\e q = q + \e$ (this is technically only defined for rational $\e$, but all our previous work goes through anyway). Then $q$ and $p$ are $\e$-interleaved if and only if $|q-p| \leq \e$. \\

A sequence $\{\tilde{q}_k\}$ is Cauchy in the context of this paper if and only if it is Cauchy in the usual sense. We can find a subsequence $\{q_k\}$ so that $q_k$ and $q_{k+1}$ are $\ds \e_{k+1}$-interleaved. Then our diagram looks like $ q_1 + \e_1 \geq q_2 + \e_2 \geq \cdots$, which has a colimit in $(\Q, \geq)$ if and only if $\{q_k\}$ converges to a rational number in the usual sense. \\

The Yoneda embedding is a functor $h_-: (\Q, \geq) \to [(\Q, \geq)^{op}, \mathbf{Set}]$. In particular, $$h_q(p) = \Hom(p,q) = \left\{\begin{array}{cl} \ast & \text{ if }p \geq q\\ \emptyset & \text{ otherwise}\end{array}\right. ,$$ and $h_q$ is a contravariant functor where $h_-$ sends the inequality $q_1 \geq q_2$ to the inclusion  $h_{q_2}(p) \supseteq h_{q_1}(p)$. Thus, $h_q$ is essentially a Dedekind cut, and if $r$ is the traditional limit of $\{q_k\}$ in $\R$, then the categorical colimit of $h_{q_1 + \e_1}\to h_{q_2 + \e_2} \to \cdots $ in $[(\Q, \geq)^{op}, \mathbf{Set}]$ is ``$h_r$'', where $$h_r(p) = \left\{\begin{array}{cl} \ast & \text{ if }p \geq r\\ \emptyset & \text{ otherwise}\end{array}\right.$$

One can show the completion $\overline{(\Q, \geq)}$ is equivalent to the category $(\R, \geq)$. However, it is not true that the two are isomorphic as categories. In particular, $(\R, \geq)$ is a skeletal category, while the completion $\overline{(\Q, \geq)}$ is not.

\subsection{Persistence Modules}

Persistence modules form the most popular category with a flow, though there are really many different categories depending on which of many different types of conditions we impose. A recent paper by Bubenik and Vergili \cite{Bub18} studies metric completeness for several of the most used categories of persistence modules (along with a host of other properties from general topology).\\

A persistence module is a functor from the order category\footnote{Notice this category is the opposite category of the one considered in the last section, i.e. $(\R, \leq) = (\R, \geq)^{op}$.} $(\R, \leq)$ to the category of $\mathbf{k}$-vector spaces $\mathbf{Vect_k}$. The category of persistence modules can be denoted $[(\R,\leq), \mathbf{Vect_k}]$. The interleaving functor $T_\e$ on a persistence module is given by 
$$(T_\e M)(a) = M(a + \e)\hspace{1.5in} (T_\e M)(a\leq b) = M(a+\e \leq b+\e)$$

\begin{example}$[(\R,\leq), \mathbf{Vect_k}]$ with the flow $T_\e$ is metrically complete.\end{example}
\begin{proof} We must show $[(\R,\leq), \mathbf{Vect_k}]$ with $T_\e$ satisfies the conditions of Corollary \ref{thm:CategorialConditionsForCompleteness}. Taking limits of functors is done pointwise; since $\mathbf{Vect_k}$ is complete, so is $[(\R,\leq), \mathbf{Vect_k}]$. Further, $T_\e$ is an autoequivalence of categories, so it preserves basically all categorical constructions, including limits. 
\end{proof}

Let $\mathbf{Vect_k^{fin}}$ be the category of finite dimensional $\mathbf{k}$-vector spaces. Then $[(\R,\leq), \mathbf{Vect_k^{fin}}]$ is the category of persistence modules $M$ where $M(a)$ is finite dimensional for each $a\in \R$. We can use the flow $T_\e$ as before because $T_\e$ preserves $[(\R,\leq), \mathbf{Vect_k^{fin}}]$.
\begin{example}$[(\R,\leq), \mathbf{Vect_k^{fin}}]$ with $T_\e$ is \emph{not} metrically complete. 
\end{example}
\begin{proof}Let $[a,b)$ be the interval module where 
$$M(s) = \left\{\begin{array}{cl}\mathbf{k} & \text{ for }s\in [a,b)\\ 0 & \text{ otherwise}\end{array}\right. \hspace{1in} M(s\leq t) = \left\{\begin{array}{cl}Id_\mathbf{k} & \text{ for }s,t\in [a,b)\\ 0 & \text{ otherwise}\end{array}\right.$$

Define the sequence of persistence modules $A_n = \bigoplus_{k=1}^n \left[\frac{-1}{k}, \frac{1}{k}\right)$. Then $\{A_n\}$ is a Cauchy sequence but has no metric limit. $\{A_n\}$ is Cauchy because for every $\e > 0$, there is an $N$ so that $T_\e A_n = T_\e A_m$ for all $n,m > N$, so the $\e$-interleaving is easy to find. It has no metric limit, because any such limit $M$ would need an infinite dimensional vector space $M(0)$. 
\end{proof}

In \cite{Bub18}, Bubenik and Vergili make similar calculations for categories of persistence modules with many different types of conditions. Note that their proof for completeness is different from ours and uses slightly different hypotheses (in general neither weaker nor stronger). 

\subsection{Generalized Persistence Modules}\label{ex:GPM}

Bubenik, de Silva, and Scott \cite{Bub15} introduced the notion of a generalized persistence module (GPM). A category of generalized persistence modules is a functor category $[\mathcal{P},\mathcal{D}]$, where $\mathcal{P}$ is a poset category and $\mathcal{D}$ is some arbitrary category. To calculate interleaving distances, they also introduced the notion of a superlinear family of translations. Then Munch, de Silva, and Stefanou \cite{MunchCatsFlow} showed how generalized persistence modules with a superlinear family of translations can be thought of a flow on the category $[\mathcal{P},\mathcal{D}]$, where the flow functors $T_\e$ are pullbacks along translations.\\

The important observation in checking metric completeness for this case is that limits in functor categories are computed pointwise. This has two corollaries:
\begin{enumerate}
\item Pulling back a functor along a translation preserves limits.
\item If $\mathcal{D}$ is a complete category, so is $[\mathcal{P},\mathcal{D}]$. 
\end{enumerate}

Therefore, we get the following:

\begin{corollary} Assume $\mathcal{P}$ is a poset category, $\mathcal{D}$ is a (categorically) complete category, and $[\mathcal{P}, \mathcal{D}]$ is the category of functors from $\mathcal{P}$ to $\mathcal{D}$. Choose a superlinear family of translations, and use it to define an interleaving distance as in \cite{Bub15} or \cite{MunchCatsFlow}. Then $[\mathcal{P}, \mathcal{D}]$ is metrically complete.
\end{corollary}

Interestingly, it is complete no matter which superlinear family of translations is chosen.

\subsection{The Derived Category of Sheaves}\label{ex:sheaves}

Recently, Kashiwara and Schapira \cite{KS} defined a flow on the derived category of sheaves of $\mathbf{k}$-vector spaces on $\R^m$, denoted $D(\mathbf{k}_{\R^m})$. We will not rehash all the details of that paper, but the broad strokes are these:\\

Let $s: \R^m \times \R^m \to \R^m$ be the addition map $s(x,y) = x+y$. Let $\mathbf{k}$ be a vector space. Let $K_\e = \mathbf{k}_{\{\|x\| \leq \e\}}$ for $\e\geq 0$. Then for an object $A\in D(\mathbf{k}_{\R^m})$, define\footnote{It deserves to be pointed out that while $T_\e$ makes sense as a functor on $D(\mathbf{k}_{\R^m})$ (without the bounded or constructible conditions) as well as on $D_c^b(\mathbf{k}_{\R^m})$ (with both the bounded and constructible conditions), it doesn't make sense on $D_c(\mathbf{k}_{\R^n})$; specifically, if $A\in D_c(\mathbf{k}_{\R^n})$ is not bounded, it is very possible that $T_\e A$ is not constructible.} $T_\e A = K_\e \star A = Rs_!(K_\e \boxtimes A)$. \\

It is shown in \cite{KS} that $T_\e$ forms a strict coflow on $D(\mathbf{k}_{\R^m})$.  \\

\begin{corollary} 
$D(\mathbf{k}_{\R^m})$ under the interleaving distance is metrically complete. 
\end{corollary}
\begin{proof}We must show that convolving with $K_\e$ preserves colimits and that $D(\mathbf{k}_{\R^m})$ has enough colimits. Convolution with $K_\e$ is an autoequivalence of categories, and so it preserves basically every categorical construction, including colimits.\\

Thus, it suffices to show that a diagram $\triangle$ of the form 
$$A_1 \to A_2 \to ... $$
has colimits in $D(\mathbf{k}_{\R^m})$. Certainly $\triangle$ has colimits in $Sh(\mathbf{k}_{\R^m})$, because $Sh(\mathbf{k}_{\R^m})$  is cocomplete. The functor $\colim_\triangle: Sh(\mathbf{k}_{\R^m})^\triangle \to Sh(\mathbf{k}_{\R^m})$ is right exact, because all colimit functors are right exact. If we can show it is left exact too, then $\colim_\triangle$ would be exact, and it would extend to a functor $\colim_\triangle : D(\mathbf{k}_{\R^m})^\triangle \to D(\mathbf{k}_{\R^m})$. This would mean $\triangle$ has colimits in $D(\mathbf{k}_{\R^m})$.\\

Therefore, it suffices to show that $\colim_\triangle$ is left exact. Say we have the exact sequence of diagrams
\begin{center}
\begin{tikzpicture}
\matrix (m) [matrix of math nodes,row sep=3em,column sep=4em,minimum width=2em]
{
0 & 0 & 0 & \cdots & \\
A_1 & A_2 & A_3 & \cdots & A\\
B_1 & B_2 & B_3 & \cdots & B\\
};
\path[-stealth]
    (m-1-1) edge node [left] { } (m-2-1)
    (m-1-2) edge node [left] { } (m-2-2)	
    (m-1-3) edge node [left] { } (m-2-3)	
    (m-2-1) edge node [left] {$p_1$} (m-3-1)
    (m-2-2) edge node [left] {$p_2$} (m-3-2)	
    (m-2-3) edge node [left] {$p_3$} (m-3-3)
    (m-2-5) edge node [left] {$p$} (m-3-5)	
    (m-2-1) edge node [above] { }(m-2-2)
    (m-2-2) edge node [above] { }(m-2-3)
    (m-2-3) edge node [above] {}(m-2-4)
    (m-2-4) edge node [above] { }(m-2-5)
    (m-3-1) edge node [above] { }(m-3-2)
    (m-3-2) edge node [above] { }(m-3-3)
    (m-3-3) edge node [above] { }(m-3-4)
    (m-3-4) edge node [above] { }(m-3-5);
            
\end{tikzpicture}
\end{center}

where $A$ and $B$ are colimits of their respective diagrams, and all the $A_k, B_k$ are sheaves. We wish to show that $p$ is injective. This amounts to showing that $p$ is injective on stalks. But because taking colimits commutes with taking stalks, we can assume the $A_k$, $B_k$, etc. in the above diagram are vector spaces over $k$. \\

Now we must show the direct limit of direct system of vector spaces is an exact functor. This is a standard result, but a proof is included here because I could not find one in the literature. We can explicitly define $A$ as 

$$A = \bigoplus_{n=1}^\infty A_n \Big/\left(a \sim \phi(a)\right)\hspace{0.5in} B = \bigoplus_{n=1}^\infty B_n \Big/\left(b \sim \psi(b)\right)$$
where $\phi$ and $\psi$ are the ``shift'' maps on $\bigoplus_{n=1}^\infty A_n$ and $\bigoplus_{n=1}^\infty B_n$, respectively. Say that $p([a]) = 0$ for some $[a]\in A$. We can represent $[a]$ which some $a\in A_n$ for some $n$, and $p([a]) = [p_n(a)] \in B$. Therefore, $[p_n(a)] = 0$, so $\psi^k(p_n(a)) = 0$ for some $k$. By commutativity of the diagram, $p_{n+k}(\phi^k(a)) = 0$. By injectivity of $p_{n+k}$, $\phi^k(a) = 0$. Therefore, $[a] = [\phi^k(a)] = 0$. This shows that $p$ is injective, which completes our proof.
\end{proof}

%\begin{remark} $D(\mathbf{k}_{\R^m})$ is an example of a category with a coflow which is metrically complete but not cocomplete.
%\end{remark}

\begin{example} As an important example, $D^b_c(\mathbf{k}_{\R^m})$ is not metrically complete. Let $F_n = \bigoplus_{k=1}^N \mathbf{1}_{[0, 2^{-k})}[-k]$. Then $\{F_n\}$ is a Cauchy sequence, but the limit does not exist in $D^b_c(\mathbf{k}_{\R^m})$. \\

Nor is it only unboundedness that is the problem. Define $C_n = \bigcup_{k=1}^n \mathbf{1}_{(2^{-k},2^{-(k-1)}]}$ and $C = \bigcup_{n=1}^\infty C_n$. Set $F_n = \mathbf{1}_{C_n}$ and $F =\mathbf{1}_C$ . Each of the $F_n \in D^b_c(\mathbf{k}_{\R^n})$ and $\lim_{n\to \infty} F_n$ exists in $D(X)$ and is bounded, but $\lim_{n\to \infty} F_n$ is not constructible. 
\end{example}

\section{Categories with a flow and Polish Spaces}\label{sec:Polish}

\footnote{This section will be more informal, but it will hopefully be useful for giving context for this paper. In particular, we will not rigorously define what we mean by separability. For the purposes of this paper, separability can mean that the induced extended quasimetric on the coskeleton is separable.}A pressing motivation for the research in this paper was in providing conditions for categories with a flow to be Polish spaces, i.e. complete and separable. Polish spaces are important because they are the foundation for many results from statistics and probability. Particularly when considering convergence of probability measures, it is important for those probability measures to be on a complete, separable space; see, for example, Prokhorov's Theorem or the result that every probability measure on a Polish space is tight \cite{Billingsley}.\\

This paper thus far has only addressed one side of that issue, that of completeness. Separability of categories with a flow seems difficult to characterize in general. Certainly specific cases are tractable, and it does seem to be related to the notion of $\kappa$-accessibility and similar concepts. However, an all-purpose, easy-to-check test remains elusive. \\

Often what happens in practice is that we have a ``little'' category which is separable sitting inside a ``big'' category which is complete, but neither is Polish. In this situation, we can find a Polish space between the two. In particular, the closure of the ``little category'' in the ``big category'' (see Definition \ref{def:Closure}) is a Polish space. \\ 

As an example, take the space of bounded, constructible sheaves $D^b_c(\mathbf{k}_{\R^m})$. This space is not complete, as shown in Section \ref{ex:sheaves}, but it is separable as long as $\mathbf{k}$ is countable, which can be shown with the help of Theorem 2.11 of \cite{KS}. On the other hand, $D(\mathbf{k}_{\R^m})$ is not separable for $m\geq 2$, but it is complete. Neither space is an appropriate space in which to do statistics, but we can use Lemma \ref{lemma:definitionOfCompletion} to create a Polish space of sheaves. \\

A similar question is ``What category of persistence modules should be used from applications?'' On the one hand, the whole category of persistence modules is metrically complete. However, in \cite{Bub18} it is shown that the collection of isomorphism classes of persistence modules is not only non-separable, but not even a set! Thus the whole category of persistence modules is too big. \\

On the other hand, there is a useful subcategory which is separable, namely the constructible persistence modules; see Remark 3.1 of \cite{Bub18}. However, this space is not complete. To see why this is a problem, consider the probability measure on persistence modules we get from looking at the persistence modules obtained from Brownian motion as was done, more or less, in \cite{Adler}. This measure is supported on persistence modules which are not constructible (nor even pointwise finite-dimensional). Thus, the category of constructible persistence modules is too little.\footnote{This was noticed in the ``decategorified'' setting of persistence diagrams by Mileyko, Mukherjee, and Harer \cite{MMH}.}\\

In fact, \cite{Bub18} considers over ten different categories of persistence modules, and none of them were both complete and separable except the ones which were trivially so, like the category of only the zero module or the category of ephemeral modules. One way to solve this problem is to take a separable subcategory and find its closure in all persistence modules using Definition \ref{def:Closure}.\\

The closure $\Cat$ of the constructible persistence modules in all persistence modules seems like a useful candidate category for using persistence modules in applications for two reasons. First, it is Polish, so we can use the relevant theorems from probability and statistics. Second, the acclaimed 1-Lipschitz map from the Stability Theorem \cite{Harer} sending bounded, continuous functions to persistence modules has its image in $\Cat$. In particular, it seems large enough to handle the expected examples of random persistence modules.

\bibliography{sources} 
\bibliographystyle{plain}

\large
Department of Mathematics, Duke University\\ 
\indent \url{joshua.cruz@duke.edu}

\end{document}